\documentclass{article}

\usepackage{amsmath}
\usepackage{amssymb}
\usepackage{amsfonts}
\usepackage{amsthm}
\usepackage{color}
\usepackage{graphicx}
\usepackage{hyperref}
\usepackage{listings}
\usepackage{pifont}
\theoremstyle{plain}
\newtheorem{thm}{Theorem}[section]
\newtheorem{lem}[thm]{Lemma}
\newtheorem{prop}[thm]{Proposition}

\theoremstyle{definition}

\theoremstyle{remark}

\makeatletter 
\@addtoreset{equation}{section}
\makeatother  

\begin{document}

\title {\bf Note on the H\"{o}lder norm estimate of the function $x\sin(1/x)$}
\author{\it Jiaqiang Mei\thanks{Project supported by NSFC(Grant No. 11171143).}, Haifeng Xu\thanks{The second author is supported by the foundation of Yangzhou University 2013CXJ006 and the Natural Science Foundation of Jiangsu Province 14KJB110027.}}
\date{\small\today}

\maketitle

\begin{abstract}
In this paper, we prove the following inequality: for any $x, y>0$, there holds
$$\big|x\sin\frac{1}{x} - y\sin\frac{1}{y} \big| \leq \sqrt{2|x - y|}.$$
\end{abstract}




Let
\[
f(x) = 
\begin{cases}
x\sin\frac{1}{x}, & x\neq 0, \\
0, & x=0.
\end{cases}
\]
It is well-known that this function $f$ is H\"{o}lder continuous with H\"{o}lder exponent $1/2$
(but not of any higher H\"{o}lder exponent, see for example [1]). For $\alpha\in(0, 1/2]$, 
the H\"{o}lder norm of $f$ is defined as
\[
|f|_{C^{0,\alpha}} = \sup_{x \neq y \in \mathbb{R}} \frac{| f(x) - f(y) |}{|x-y|^\alpha}.
\]
In this paper, we investigate the H\"{o}lder norm and provide the following estimate 
\[
|f|_{C^{0,1/2}} \leq \sqrt{2}.
\]

Since $f$ is an even function, in the sequel we will assume $x, y>0$. The method we shall 
use to prove the above estimate is rather elementary, however, it turns out that the argument is a little bit delicate
in several situations. Our argument roughly runs as follows: firstly we investigate the 
monotonicity property of $f$ on $(0, \infty)$. It is easy to show that $(0, \infty)$ is divided
by a sequence of consecutive intervals, each of which contains exactly one inflection point
of the form $\frac{1}{n\pi}$. Secondly we study the H\"{o}lder continuity property of $f$ near
$0$. We use certain integrals to deal with the oscillation phenomenon of $f$ near $0$. Although the H\"{o}lder
norm near infinity is not difficult to estimate, there are two intermediate intervals in which
more delicate analysis will be involved.

\section{Preliminaries}

Let $\varphi(t)=\sin t-t\cos t$, $t\in(0,\infty)$. It's easy to see that for each $n\geqslant 1$, there is only one solution $\alpha_n\in(n\pi,n\pi+\frac{\pi}{2})$ for the equation $\varphi(t)=0$. Let $\alpha_n=n\pi+\frac{\pi}{2}-\theta_n$, $\theta_n\in(0,\frac{\pi}{2})$, we have

\begin{lem}\label{lem:1.1}
For each $n\geqslant 1$, the following estimates hold for $\theta_n$,
\begin{eqnarray}
& &\theta_n<\frac{1}{\alpha_n}<\frac{1}{n\pi},\label{eqn:1.1}\\
& &\theta_n<\frac{1}{n\pi+\frac{\pi}{2}}(1+\theta_n^2),\label{eqn:1.2}\\
& &\theta_n<\frac{2n+1}{4}\pi-\sqrt{(\frac{2n+1}{4}\pi)^2-1}.\label{eqn:1.3}
\end{eqnarray}
\end{lem}
\begin{proof}
Since $\alpha_n$ is a solution for the equation $\varphi(t)=0$, we have
\[
\sin\alpha_n=\alpha_n\cdot\cos\alpha_n.
\]
Using $\alpha_n=n\pi+\frac{\pi}{2}-\theta_n$, we have
\begin{equation}\label{eqn:1.4}
1=\alpha_n\cdot\tan\theta_n > \alpha_n\cdot\theta_n,
\end{equation}
which yields \eqref{eqn:1.1}. Substitute $\alpha_n=n\pi+\frac{\pi}{2}-\theta_n$ in the above inequality, one gets \eqref{eqn:1.2} and \eqref{eqn:1.3}.
\end{proof}

{\bf Remark}. For $\theta_1$, it will be more convenient to use the deduced estimate $\theta_1< \frac{\pi}{14}$.

\begin{lem}\label{lem:1.2}
For each $n\geqslant 1$, we have the following estimate
\[
\theta_n>\sin\theta_n>\frac{1}{n\pi+\frac{\pi}{2}}.
\]
\end{lem}
\begin{proof}
Let $\beta_n=n\pi+\frac{\pi}{2}-\eta_n$, where $\eta_n\in(0,\frac{\pi}{2})$ satisfies
\[
\sin\eta_n=\frac{1}{n\pi+\frac{\pi}{2}}.
\]
We have
\[
\begin{split}
(-1)^n\cdot\varphi(\beta_n)&=\cos\eta_n-(n\pi+\frac{\pi}{2}-\eta_n)\cdot\sin\eta_n\\
&=\cos\eta_n+\eta_n\cdot\sin\eta_n-1\\
&>\cos\eta_n+\sin^2\eta_n-1\\
&=\cos\eta_n-\cos^2\eta_n>0.
\end{split}
\]
Since $(-1)^n\cdot\varphi(t)$ is monotonically decreasing in $(n\pi,n\pi+\frac{\pi}{2})$ and $\alpha_n$ is the only solution for the equation $\varphi(t)=0$, we have $\beta_n>\alpha_n$. Thus
\[
\theta_n>\eta_n,\quad\sin\theta_n>\sin\eta_n=\frac{1}{n\pi+\frac{\pi}{2}}.
\]
More precisely, we have
\begin{equation}\label{eqn:1.4-5}
\theta_n>\arcsin\frac{1}{n\pi+\frac{\pi}{2}}.
\end{equation}
\end{proof}

\begin{lem}\label{lem:1.3}
For each $n\geqslant 1$, we have the following estimate
\begin{equation}\label{eqn:1.5}
0<\theta_n-\theta_{n+1}<\frac{\pi}{\alpha_n\cdot\alpha_{n+1}}.
\end{equation}
\end{lem}
\begin{proof}
By \eqref{eqn:1.4} we have
\[
\tan\theta_n=\frac{1}{\alpha_n}>\frac{1}{\alpha_{n+1}}=\tan\theta_{n+1},
\]
which yields $\theta_n>\theta_{n+1}$.

On the other hand,
\[
\begin{split}
\tan(\theta_n-\theta_{n+1})&=\frac{\tan\theta_n-\tan\theta_{n+1}}{1+\tan\theta_n\cdot\tan\theta_{n+1}}=\frac{\frac{1}{\alpha_n}-\frac{1}{\alpha_{n+1}}}{1+\frac{1}{\alpha_n\cdot\alpha_{n+1}}}\\
&=\frac{\alpha_{n+1}-\alpha_n}{1+\alpha_n\cdot\alpha_{n+1}}=\frac{\pi+\theta_n-\theta_{n+1}}{1+\alpha_n\cdot\alpha_{n+1}}.
\end{split}
\]
By using $\theta_n-\theta_{n+1}<\tan(\theta_n-\theta_{n+1})$ we can deduce the right hand side of \eqref{eqn:1.5}.
\end{proof}

\begin{lem}\label{lem:1.4}
For $n\geqslant 1$, we define the constant $C_n$ as follows
\[
C_n=\frac{(\alpha_{n+1}-\alpha_n)^2}{\pi^2\alpha_n^2\alpha_{n+1}^2}\biggl[\frac{1}{10}(\alpha_{n+1}^5-\alpha_n^5)+\frac{1}{4}(\alpha_{n+1}-\alpha_n)\Bigl(1+\frac{\alpha_{n+1}\alpha_n-1}{(1+\alpha_{n+1}^2)(1+\alpha_n^2)}\Bigr)\biggr].
\]
Then $C_n<2$ for each $n>1$ and $C_1<2.26$.
\end{lem}
\begin{proof}
Let $\delta_n=\alpha_{n+1}-\alpha_n=\pi+\theta_n-\theta_{n+1}$. We have $\delta_n<\pi(1+\frac{1}{\alpha_n\alpha_{n+1}})$ according to Lemma \ref{lem:1.3} and
\begin{equation}\label{eqn:1.6}
\begin{split}
\alpha_{n+1}^5-\alpha_n^5&=(\alpha_n+\delta_n)^5-\alpha_n^5\\
&=5\alpha_n^4\delta_n+10\alpha_n^3\delta_n^2+10\alpha_n^2\delta_n^3+5\alpha_n\delta_n^4+\delta_n^5\\
&=5\alpha_n^2\alpha_{n+1}^2\delta_n+5\alpha_n\alpha_{n+1}\delta_n^3+\delta_n^5.
\end{split}
\end{equation}
Thus, if let
\[
G_n=\frac{(\alpha_{n+1}-\alpha_n)^2}{\pi^2\alpha_n^2\alpha_{n+1}^2}\cdot\frac{1}{10}(\alpha_{n+1}^5-\alpha_n^5),
\]
then,
\[
\begin{split}
G_n&=\frac{\delta_n^2}{\pi^2\alpha_n^2\alpha_{n+1}^2}\biggl[\frac{1}{10}(5\alpha_n^2\alpha_{n+1}^2\delta_n+5\alpha_n\alpha_{n+1}\delta_n^3+\delta_n^5)\biggr]\\
&=\frac{\delta_n^3}{2\pi^2}\biggl[1+\frac{\delta_n^2}{\alpha_n\alpha_{n+1}}+\frac{1}{5}\cdot\frac{\delta_n^4}{\alpha_n^2\alpha_{n+1}^2}\biggr].
\end{split}
\]
By Lemma \ref{lem:1.1}, we have the following estimate
\[
\alpha_1\alpha_2=(\frac{3\pi}{2}-\theta_1)(\frac{5\pi}{2}-\theta_2)>\frac{7}{2}\pi^2.
\]
Recall that $\alpha_n=n\pi+\frac{\pi}{2}-\theta_n$, where $\theta_n\in(0,\frac{\pi}{2})$. If $n>1$, then by Lemma \ref{lem:1.1}, we have
\[
\begin{split}
\alpha_n\alpha_{n+1}&\geqslant\alpha_2\alpha_3=(\frac{5\pi}{2}-\theta_2)(\frac{7\pi}{2}-\theta_3)\\
&>\biggl[\frac{5\pi}{2}-\Bigl(\frac{5\pi}{4}-\sqrt{(\frac{5\pi}{4})^2-1}\Bigr)\biggr]\cdot\biggl[\frac{7\pi}{2}-\Bigl(\frac{7\pi}{4}-\sqrt{(\frac{7\pi}{4})^2-1}\Bigr)\biggr]\\
&=\biggl[\frac{5\pi}{4}+\sqrt{(\frac{5\pi}{4})^2-1}\biggr]\cdot\biggl[\frac{7\pi}{4}+\sqrt{(\frac{7\pi}{4})^2-1}\biggr]\\
&>7.7245\cdot 10.9038\\
&>84.22.
\end{split}
\]
Then
\[
\frac{\delta_n^2}{\alpha_n\alpha_{n+1}}<\frac{\pi^2(1+\frac{1}{\alpha_n\alpha_{n+1}})^2}{\alpha_n\alpha_{n+1}}<\frac{\pi^2(1+\frac{1}{84.22})^2}{84.22}<0.12.
\]
Thus we can estimate  $G_n$ $(n>1)$ as follows
\[
\begin{split}
G_n&<\frac{1}{2\pi^2}\cdot\pi^3(1+\frac{1}{\alpha_n\alpha_{n+1}})^3\biggl[1+\frac{\delta_n^2}{\alpha_n\alpha_{n+1}}+\frac{1}{5}\cdot\frac{\delta_n^4}{\alpha_n^2\alpha_{n+1}^2}\biggr]\\
&<\frac{\pi}{2}(1+\frac{1}{84.22})^3\Bigl[1+0.12+\frac{1}{5}\cdot 0.12^2\Bigr]\\
&<1.83.
\end{split}
\]
Similarly, for $n=1$, we have
\[
\begin{split}
\alpha_1\alpha_2&=(\frac{3\pi}{2}-\theta_1)(\frac{5\pi}{2}-\theta_2)\\
&>\biggl[\frac{3\pi}{2}-\Bigl(\frac{3\pi}{4}-\sqrt{(\frac{3\pi}{4})^2-1}\Bigr)\biggr]\cdot\biggl[\frac{5\pi}{2}-\Bigl(\frac{5\pi}{4}-\sqrt{(\frac{5\pi}{4})^2-1}\Bigr)\biggr]\\
&=\biggl[\frac{3\pi}{4}+\sqrt{(\frac{3\pi}{4})^2-1}\biggr]\cdot\biggl[\frac{5\pi}{4}+\sqrt{(\frac{5\pi}{4})^2-1}\biggr]\\
&>4.4896\cdot 7.7245\\
&>34.6.
\end{split}
\]
Thus,
\[
\frac{\delta_1^2}{\alpha_1\alpha_2}<\frac{\pi^2(1+\frac{1}{\alpha_1\alpha_2})^2}{\alpha_1\alpha_2}<\frac{\pi^2}{34.6}(1+\frac{1}{34.6})^2<0.302.
\]
Then, we get the estimate of $G_1$
\[
\begin{split}
G_1&<\frac{1}{2\pi^2}\cdot\pi^3(1+\frac{1}{\alpha_1\alpha_2})^3\biggl[1+\frac{\delta_1^2}{\alpha_1\alpha_2}+\frac{1}{5}\cdot\frac{\delta_1^4}{\alpha_1^2\alpha_2^2}\biggr]\\
&<\frac{\pi}{2}(1+\frac{1}{34.6})^3\Bigl[1+0.302+\frac{1}{5}\cdot 0.302^2\Bigr]\\
&<2.259.
\end{split}
\]
If let
\[
F_n=1+\frac{\alpha_{n+1}\alpha_n-1}{(1+\alpha_{n+1}^2)(1+\alpha_n^2)},
\]
then
\[
\begin{split}
F_n&=1+\frac{\alpha_{n+1}\alpha_n-1}{1+\alpha_n^2+\alpha_{n+1}^2+\alpha_{n+1}^2\alpha_n^2}<1+\frac{\alpha_{n+1}\alpha_n-1}{1+2\alpha_n\alpha_{n+1}+\alpha_{n+1}^2\alpha_n^2}\\
&<1+\frac{\alpha_{n+1}\alpha_n+1}{(1+\alpha_n\alpha_{n+1})^2}=\frac{2+\alpha_n\alpha_{n+1}}{1+\alpha_n\alpha_{n+1}}.
\end{split}
\]
Thus,
\[
\frac{1}{4}\delta_n F_n<\frac{1}{4}\pi(1+\frac{1}{\alpha_n\alpha_{n+1}})\cdot\frac{2+\alpha_n\alpha_{n+1}}{1+\alpha_n\alpha_{n+1}}=\frac{\pi}{4}(1+\frac{2}{\alpha_n\alpha_{n+1}}).
\]
Therefore,
\[
\begin{split}
\frac{\delta_n^2}{\pi^2\alpha_n^2\alpha_{n+1}^2}\cdot\frac{1}{4}\delta_n F_n&<\frac{\pi^2(1+\frac{1}{\alpha_n\alpha_{n+1}})^2}{\pi^2\alpha_n^2\alpha_{n+1}^2}\cdot\frac{\pi}{4}(1+\frac{2}{\alpha_n\alpha_{n+1}})\\
&=\frac{\pi}{4}\cdot\frac{1}{(\alpha_n\alpha_{n+1})^2}\cdot(1+\frac{1}{\alpha_n\alpha_{n+1}})^2(1+\frac{2}{\alpha_n\alpha_{n+1}}).
\end{split}
\]
Using the above estimates
\[
\alpha_n\alpha_{n+1}>
\begin{cases}
84.22, & n>1,\\
34.6, & n=1,
\end{cases}
\]
we have
\[
\frac{\delta_n^2}{\pi^2\alpha_n^2\alpha_{n+1}^2}\cdot\frac{1}{4}\delta_n F_n <
\begin{cases}
0.00012, & n>1,\\
0.00080, & n=1.
\end{cases}
\]
Therefore, we obtain the following estimate of $C_n$:
\[
C_n=G_n+\frac{\delta_n^2}{\pi^2\alpha_n^2\alpha_{n+1}^2}\cdot\frac{1}{4}\delta_n F_n <
\begin{cases}
1.83+0.00012=1.83012, & n>1,\\
2.259+0.00080=2.25980, & n=1.
\end{cases}
\]
\end{proof}

\begin{lem}\label{lem:1.5}
For $\theta\in(0,\frac{\pi}{2})$, we have
\[
\sin\theta-\theta\cos\theta<\frac{1}{3}\theta^3.
\]
\end{lem}
\begin{proof}
Let $p(\theta)=\sin\theta-\theta\cos\theta-\frac{1}{3}\theta^3$, $\theta\in(0,\frac{\pi}{2})$. We have
\[
p'(\theta)=\theta\sin\theta-\theta^2<0,
\]
thus $p(\theta)<p(0)=0$.
\end{proof}

In next section, we will also need the following well-known inequality([2,3]).
\begin{lem}[Wirtinger's inequality]\label{lem:1.6}
Suppose $g\in C^1[a,b]$, $g(a)=g(b)=0$. Then
\[
\int_a^b g^2(t)dt\leqslant (\frac{b-a}{\pi})^2\cdot\int_a^b|g'(t)|^2dt.
\]
\end{lem}


\section{H\"{o}lder properties of $f(x)$}

In this section, we shall study H\"{o}lder properties of $f(x)=x\cdot\sin\frac{1}{x}$. Simple calculation gives
\begin{equation}\label{eqn:2.1}
f'(x)=\sin\frac{1}{x}-\frac{1}{x}\cos\frac{1}{x},\quad f''(x)=-\frac{1}{x^3}\sin\frac{1}{x}.
\end{equation}
Thus $f(x)$ is monotone in each interval $[\frac{1}{\alpha_{n+1}},\frac{1}{\alpha_n}]$.

\begin{prop}\label{prop:2.1}
For $x,y\in[\frac{1}{\alpha_{n+1}},\frac{1}{\alpha_n}]$, $(n>1)$, we have
\begin{equation}\label{eqn:2.2}
|f(y)-f(x)|\leqslant\sqrt{2|y-x|}.
\end{equation}
For $x,y\in[\frac{1}{\alpha_2},\frac{1}{\alpha_1}]$, we have
\begin{equation}\label{eqn:2.2-2}
|f(y)-f(x)|\leqslant\sqrt{2.26|y-x|}.
\end{equation}
\end{prop}
\begin{proof}
By Cauchy's inequality, for $x,y\in[\frac{1}{\alpha_{n+1}},\frac{1}{\alpha_n}]$ we have
\[
\begin{split}
|f(y)-f(x)|^2&=\biggl|\int_x^y f'(t)dt\biggr|^2\\
&\leqslant |y-x|\cdot\biggl|\int_x^y |f'(t)|^2 dt\biggr|\\
&\leqslant |y-x|\cdot\int_{\frac{1}{\alpha_{n+1}}}^{\frac{1}{\alpha_n}} |f'(t)|^2 dt.\\
\end{split}
\]
Since $f'(\frac{1}{\alpha_{n+1}})=f'(\frac{1}{\alpha_{n}})=0$, by Wirtinger's inequality, we have
\begin{equation}\label{eqn:2.3}
\begin{split}
|f(y)-f(x)|^2&\leqslant |y-x|\cdot\frac{1}{\pi^2}(\frac{1}{\alpha_{n}}-\frac{1}{\alpha_{n+1}})^2\cdot\int_{\frac{1}{\alpha_{n+1}}}^{\frac{1}{\alpha_n}} |f''(t)|^2 dt\\
&=|y-x|\cdot\frac{1}{\pi^2}(\frac{1}{\alpha_{n}}-\frac{1}{\alpha_{n+1}})^2\cdot\int_{\alpha_n}^{\alpha_{n+1}}u^4\sin^2 u du.
\end{split}
\end{equation}
Denote by
\[
I_n=\int_{\alpha_n}^{\alpha_{n+1}}u^4\cdot\sin^2 u\,du.
\]
Direct calculation yields
\[
\begin{split}
I_n&=\frac{1}{10}(\alpha_{n+1}^5-\alpha_n^5)-\frac{1}{4}\bigl[\alpha_{n+1}^4\sin(2\theta_{n+1})-\alpha_{n}^4\sin(2\theta_{n})\bigr]\\
&\quad +\frac{1}{2}\bigl[\alpha_{n+1}^3\cos(2\theta_{n+1})-\alpha_n^3\cos(2\theta_n)\bigr]+\frac{3}{4}\bigl[\alpha_{n+1}^2\sin(2\theta_{n+1})-\alpha_n^2\sin(2\theta_n)\bigr]\\
&\quad -\frac{3}{4}\bigl[\alpha_{n+1}\cos(2\theta_{n+1})-\alpha_n\cos(2\theta_n)\bigr]-\frac{3}{8}\bigl[\sin(2\theta_{n+1})-\sin(2\theta_n)\bigr].
\end{split}
\]
Notice that
\[
\alpha_n\sin(2\theta_n)=\frac{1}{\tan\theta_n}\cdot\sin(2\theta_n)=2\cos^2\theta_n.
\]
Using this, we can write $I_n$ as follows
\[
\begin{split}
I_n&=\frac{1}{10}(\alpha_{n+1}^5-\alpha_n^5)-\frac{1}{4}\bigl[\alpha_{n+1}^3\cdot 2\cos^2\theta_{n+1}-\alpha_{n}^3\cdot 2\cos^2\theta_{n}\bigr]\\
&\quad +\frac{1}{2}\bigl[\alpha_{n+1}^3\cos(2\theta_{n+1})-\alpha_n^3\cos(2\theta_n)\bigr]+\frac{3}{4}\bigl[\alpha_{n+1}\cdot 2\cos^2\theta_{n+1}-\alpha_n\cdot 2\cos^2\theta_n\bigr]\\
&\quad -\frac{3}{4}\bigl[\alpha_{n+1}\cos(2\theta_{n+1})-\alpha_n\cos(2\theta_n)\bigr]-\frac{3}{8}\bigl[\sin(2\theta_{n+1})-\sin(2\theta_n)\bigr]\\
&=\frac{1}{10}(\alpha_{n+1}^5-\alpha_n^5)-\frac{1}{2}\bigl[\alpha_{n+1}^3\sin^2\theta_{n+1}-\alpha_n^3\sin^2\theta_n\bigr]+\frac{3}{4}(\alpha_{n+1}-\alpha_n)\\
&\quad -\frac{3}{8}\bigl[\sin(2\theta_{n+1})-\sin(2\theta_n)\bigr].
\end{split}
\]
Since $\alpha_k=\frac{1}{\tan\theta_k}$, we have
\[
\alpha_k^3\sin^2\theta_k=\alpha_k\cdot\cos^2\theta_k,\qquad\cos^2\theta_k=\frac{\alpha_k^2}{1+\alpha_k^2},\qquad\sin(2\theta_k)=\frac{2\alpha_k}{1+\alpha_k^2}.
\]
Then if let $D=\alpha_{n+1}^3\sin^2\theta_{n+1}-\alpha_n^3\sin^2\theta_n$ and $E=\sin(2\theta_{n+1})-\sin(2\theta_n)$, we will have
\[
\begin{split}
D&=\alpha_{n+1}\cos^2\theta_{n+1}-\alpha_n\cos^2\theta_n\\
&=\frac{\alpha_{n+1}^3}{1+\alpha_{n+1}^2}-\frac{\alpha_n^3}{1+\alpha_n^2}\\
&=(\alpha_{n+1}-\alpha_n)+\frac{(\alpha_{n+1}-\alpha_n)(\alpha_{n+1}\alpha_n-1)}{(1+\alpha_{n+1}^2)(1+\alpha_n^2)}
\end{split}
\]
and
\[
E=\frac{2\alpha_{n+1}}{1+\alpha_{n+1}^2}-\frac{2\alpha_{n}}{1+\alpha_{n}^2}=\frac{2(\alpha_{n+1}-\alpha_n)(1-\alpha_{n}\alpha_{n+1})}{(1+\alpha_{n+1}^2)(1+\alpha_n^2)}.
\]
Substitute them in the last equation of $I_n$, we get
\[
I_n=\frac{1}{10}(\alpha_{n+1}^5-\alpha_n^5)+\frac{1}{4}(\alpha_{n+1}-\alpha_n)\biggl[1+\frac{\alpha_{n+1}\alpha_n-1}{(1+\alpha_{n+1}^2)(1+\alpha_n^2)}\biggr].
\]
By \eqref{eqn:2.3} and Lemma \ref{lem:1.4}, we have
\[
|f(y)-f(x)|^2\leqslant C_n|y-x| \leqslant
\begin{cases}
2|y-x|,& n>1,\\
2.26|y-x|, & n=1.
\end{cases}
\]
for any $x,y\in[\frac{1}{\alpha_{n+1}},\frac{1}{\alpha_n}]$.
\end{proof}

For $n=1$, we now improve the above estimate.

\begin{prop}\label{prop:2.2}
For $x,y\in[\frac{1}{\alpha_2},\frac{1}{\alpha_1}]$, we also have
\begin{equation}\label{eqn:2.5}
|f(y)-f(x)|\leqslant\sqrt{2|y-x|}.
\end{equation}
\end{prop}
\begin{proof}
We consider the following function
\[
\varphi(x,y)=\frac{f(x)-f(y)}{\sqrt{y-x}},\quad x,y\in\big[\frac{1}{\alpha_2},\frac{1}{\alpha_1}\big],\ y>x.
\]
By Proposition \ref{prop:2.1}, $0<\varphi(x,y)<\sqrt{2.26}$. Since $\max\limits_{[\frac{1}{\alpha_2},\frac{1}{\alpha_1}]}|f'(x)|=f'(\frac{1}{2\pi})=2\pi$, we also have
\[
\varphi(x,y)\leqslant\frac{2\pi\cdot(y-x)}{\sqrt{y-x}}=2\pi\cdot\sqrt{y-x}.
\]
Suppose $\varphi(x,y)$ attains its maximum at $(x_0,y_0)$, then the above estimate implies that $x_0<y_0$. Using $f'(\frac{1}{\alpha_1})=f'(\frac{1}{\alpha_2})=0$, we have
\[
\partial_x\varphi(\frac{1}{\alpha_2},y)=\frac{1}{2}\cdot\frac{f(\frac{1}{\alpha_2})-f(y)}{(y-\frac{1}{\alpha_2})^{3/2}}>0,\quad\partial_y\varphi(x,\frac{1}{\alpha_1})=-\frac{1}{2}\cdot\frac{f(x)-f(\frac{1}{\alpha_1})}{(\frac{1}{\alpha_1}-x)^{3/2}}<0.
\]
It follows that $(x_0,y_0)$ is an interior point, thus
\[
\partial_x\varphi(x_0,y_0)=\partial_y\varphi(x_0,y_0)=0.
\]
Simple calculation yields
\begin{equation}\label{eqn:2.6}
f'(x_0)=f'(y_0)=\frac{1}{2}\cdot\frac{f(y_0)-f(x_0)}{y_0-x_0}.
\end{equation}
From \eqref{eqn:2.6} one can deduce the following facts
\begin{equation}\label{eqn:2.7}
(\text{i})\ \ x_0<\frac{1}{2\pi}<y_0;\quad (\text{ii})\ \ |f'(x_0)|=|f'(y_0)|\leqslant\pi.
\end{equation}
Using \eqref{eqn:2.7} we can also deduce an upper bound for $x_0$. In fact,
\[
\Bigl|f'(\frac{4}{9\pi})\Bigr|=\biggl|\sin\frac{9}{4}\pi-\frac{9}{4}\pi\cos\frac{9}{4}\pi\biggr|=\frac{\sqrt{2}}{2}(\frac{9}{4}\pi-1)>\pi,
\]
this implies that $x_0<\frac{4}{9\pi}$.

Assume that $\varphi(x_0,y_0)\geqslant\sqrt{2}$, we will derive a contradiction. By \eqref{eqn:2.6} we have
\begin{equation}\label{eqn:2.8}
|f'(x_0)|=|f'(y_0)|\geqslant\frac{1}{\sqrt{2(y_0-x_0)}}\geqslant\frac{1}{|f(x_0)-f(y_0)|}.
\end{equation}
By Lemma \ref{lem:1.3}--\ref{lem:1.4},
\[
\begin{split}
0<y_0-x_0&<\frac{1}{\alpha_1}-\frac{1}{\alpha_2}=\frac{1}{\alpha_1\alpha_2}(\pi+\theta_1-\theta_2)\\
&<\frac{\pi}{\alpha_1\alpha_2}(1+\frac{1}{\alpha_1\alpha_2})<0.1.
\end{split}
\]
From \eqref{eqn:2.8} it follows that
\[
|f'(y_0)|\geqslant\sqrt{5}>1=\Bigl|f'(\frac{2}{3\pi})\Bigr|,
\]
which yields $y_0<\frac{2}{3\pi}$. By \eqref{eqn:2.8} again we have
\[
|f'(y_0)|\geqslant\frac{1}{|f(x_0) - f(y_0)|}\geqslant\frac{1}{x_0+y_0}\geqslant\frac{1}{\frac{4}{9\pi}+\frac{2}{3\pi}}=\frac{9}{10}\pi.
\]
Since
\[
\begin{split}
\Bigl|f'(\frac{1}{\frac{3\pi}{2}+\frac{1}{3}})\Bigr|&=\biggl|-\cos\frac{1}{3}-(\frac{3\pi}{2}+\frac{1}{3})\sin\frac{1}{3}\biggr|\\
&< 1+(\frac{3\pi}{2}+\frac{1}{3})\cdot\frac{1}{3}<\frac{9}{10}\pi,
\end{split}
\]
we have $y_0<\frac{1}{\frac{3\pi}{2}+\frac{1}{3}}$. We can also deduce a lower bound for $x_0$. In fact,
\[
\Bigl|f'(\frac{1}{2\pi+\frac{\pi}{3}})\Bigr|=\biggl|\frac{\sqrt{3}}{2}-(2\pi+\frac{\pi}{3})\cdot\frac{1}{2}\biggr|=\frac{7}{6}\pi-\frac{\sqrt{3}}{2}<\frac{9}{10}\pi,
\]
thus $x_0>\frac{1}{2\pi+\frac{\pi}{3}}$, and $f(x_0)<\frac{1}{2\pi+\frac{\pi}{3}}\cdot\frac{\sqrt{3}}{2}$. It follows that
\[
\begin{split}
|f(y_0)-f(x_0)|&\leqslant y_0+f(x_0)\\
&<\frac{1}{\frac{3\pi}{2}+\frac{1}{3}}+\frac{1}{2\pi+\frac{\pi}{3}}\cdot\frac{\sqrt{3}}{2}\\
&<\frac{1}{\pi}.
\end{split}
\]
Using \eqref{eqn:2.8} again, we have $|f'(x_0)|=|f'(y_0)|>\pi$, which contradicts with \eqref{eqn:2.7}.
\end{proof}

\begin{prop}\label{prop:2.3}
For $x,y\in [\frac{1}{\pi},\infty)$, we have
\begin{equation}\label{eqn:2.9}
|f(y)-f(x)|\leqslant\sqrt{2|y-x|}.
\end{equation}
\end{prop}
\begin{proof}
From \eqref{eqn:2.1} we see that $f$ is concave in $[\frac{1}{\pi},\infty)$, and $\max_{[\frac{1}{\pi},\infty)}f'=\pi$. First, we prove the following inequality
\begin{equation}\label{eqn:2.10}
f(x)\leqslant\sqrt{2\big(x-\frac{1}{\pi}\big)},\quad\forall\ x\in\big[\frac{1}{\pi},\infty).
\end{equation}
In fact, if $x\in[\frac{1}{\pi},\frac{1}{\pi}+\frac{2}{\pi^2}]$, by the mean value theorem, we have
\[
f(x)=f(x)-f(\frac{1}{\pi})=f'(\xi)\cdot(x-\frac{1}{\pi})\leqslant\pi(x-\frac{1}{\pi})\leqslant\sqrt{2\big(x-\frac{1}{\pi}\big)}.
\]
If $x\in(\frac{1}{\pi}+\frac{2}{\pi^2},\frac{1}{\pi}+\frac{1}{2}]$, then it's easy to verify that
\[
x<\sqrt{2\big(x-\frac{1}{\pi}\big)},
\]
thus
\[
f(x)<x<\sqrt{2\big(x-\frac{1}{\pi}\big)}.
\]
If $x>\frac{1}{\pi}+\frac{1}{2}$, then
\[
f(x)<1<\sqrt{2\big(x-\frac{1}{\pi}\big)}.
\]
Second, let $y>x\geqslant\frac{1}{\pi}$. Since $f$ is concave, we have
\[
0<f(y)-f(x)\leqslant f(y-x+\frac{1}{\pi})-f(\frac{1}{\pi})=f(y-x+\frac{1}{\pi}).
\]
By \eqref{eqn:2.10}, it follows that
\[
0<f(y)-f(x)\leqslant\sqrt{2(y-x)}.
\]
\end{proof}

\begin{prop}\label{prop:2.4}
For $x,y\in[\frac{1}{\alpha_1},\infty)$, we have
\begin{equation}\label{eqn:2.11}
|f(y)-f(x)|\leqslant\sqrt{2|y-x|}.
\end{equation}
\end{prop}
\begin{proof}
Consider the function
\[
\psi(x,y)=\frac{f(y)-f(x)}{\sqrt{y-x}},\quad x,y\in\big[\frac{1}{\alpha_1},\infty),\ y>x.
\]
Since $f$ is increasing in $[\frac{1}{\alpha_1},\infty)$, and $\max\limits_{[\frac{1}{\alpha_1},\infty)}f'=f'(\frac{1}{\pi})=\pi$, we have
\[
0<\psi(x,y)\leqslant\pi\cdot\sqrt{y-x}.
\]
If $x\geqslant\frac{1}{\pi}$, by Proposition \ref{prop:2.3} we have $\psi(x,y)\leqslant\sqrt{2}$. If $x\in[\frac{1}{\alpha_1},\frac{1}{\pi})$, $y\geqslant\frac{4}{\pi}$, then
\[
0<f(y)-f(x)\leqslant 1+\frac{1}{\alpha_1}\sin\alpha_1=1+\sin\theta_1<1+\theta_1.
\]
By Lemma \ref{lem:1.1}, it's easy to verify that
\[
(1+\theta_1)^2<\frac{6}{\pi}.
\]
It follows that
\[
0<f(y)-f(x)<\sqrt{\frac{6}{\pi}}\leqslant\sqrt{2(y-x)}.
\]
It remains to consider the case when $x\in[\frac{1}{\alpha_1},\frac{1}{\pi}]$, $x<y\leqslant\frac{4}{\pi}$. Suppose $\psi(x,y)$ attains its maximum at $(x_0,y_0)$. Assume that $\psi(x_0,y_0)\geqslant\sqrt{2}$, we will derive a contradiction. Using $f'(\frac{1}{\alpha_1})=0$, we have
\[
\partial_x\psi(\frac{1}{\alpha_1},y)=\frac{1}{2}\cdot\frac{f(y)-f(\frac{1}{\alpha_1})}{(y-\frac{1}{\alpha_1})^{3/2}}>0.
\]
Thus $(x_0,y_0)$ must be an interior point. It follows that
\begin{equation}\label{eqn:2.12}
f'(x_0)=f'(y_0)=\frac{1}{2}\cdot\frac{f(y_0)-f(x_0)}{y_0-x_0}.
\end{equation}
From \eqref{eqn:2.12} and the mean value theorem, we have
\begin{equation}\label{eqn:2.13}
0<f'(x_0)=f'(y_0)\leqslant\frac{\pi}{2}.
\end{equation}
Using \eqref{eqn:2.13} we can derive an upper bound for $x_0$ and a lower bound for $y_0$. In fact,
\[
f'(\frac{4}{5\pi})=\sin\frac{5}{4}\pi-\frac{5}{4}\pi\cos\frac{5}{4}\pi=-\frac{\sqrt{2}}{2}+\frac{5\pi}{4}\cdot\frac{\sqrt{2}}{2}>\frac{\pi}{2}.
\]
So $x_0<\frac{4}{5\pi}$. Similarly,
\[
f'(\frac{13}{8\pi})=\sin\frac{8}{13}\pi-\frac{8}{13}\pi\cos\frac{8}{13}\pi=\cos\frac{3}{26}\pi+\frac{8}{13}\pi\sin\frac{3}{26}\pi.
\]
Using Taylor's expansion formula, it's easy to get the estimate $f'(\frac{13}{8\pi})>\frac{\pi}{2}$. Thus $y_0>\frac{13}{8\pi}$. We also have
\[
\begin{split}
f(\frac{13}{8\pi})-f(\frac{4}{5\pi})&=\frac{13}{8\pi}\sin\frac{8}{13}\pi+\frac{4}{5\pi}\sin\frac{\pi}{4}\\
&=\frac{13}{8\pi}\cos\frac{3}{26}\pi+\frac{2\sqrt{2}}{5\pi}\\
&<\frac{2.1}{\pi}.
\end{split}
\]
On the other hand,
\[
2.6\cdot(\frac{13}{8\pi}-\frac{4}{5\pi})=\frac{2.145}{\pi},
\]
hence we have
\[
f(\frac{13}{8\pi})-f(\frac{4}{5\pi})<2.6\cdot(\frac{13}{8\pi}-\frac{4}{5\pi}).
\]
By the mean value theorem, we have
\[
\begin{split}
f(y_0)-f(x_0)&=f(y_0)-f(\frac{13}{8\pi})+f(\frac{13}{8\pi})-f(\frac{4}{5\pi})+f(\frac{4}{5\pi})-f(x_0)\\
&=f'(\xi)\cdot(y_0-\frac{13}{8\pi})+f(\frac{13}{8\pi})-f(\frac{4}{5\pi})+f'(\eta)\cdot(\frac{4}{5\pi}-x_0)\\
&\leqslant f'(\frac{13}{8\pi})\cdot(y_0-\frac{13}{8\pi})+2.6\cdot(\frac{13}{8\pi}-\frac{4}{5\pi})+f'(\frac{4}{5\pi})\cdot(\frac{4}{5\pi}-x_0)\\
&<2.6\cdot(y_0-x_0).
\end{split}
\]
By \eqref{eqn:2.12} it follows that $f'(x_0)=f'(y_0)<1.3$. Using this, we can improve the lower bound for $y_0$. In fact,
\[
\begin{split}
f'(\frac{7}{4\pi})&=\sin\frac{4\pi}{7}-\frac{4\pi}{7}\cos\frac{4\pi}{7} \\
                  &=\cos\frac{\pi}{14} + \frac{4\pi}{7}\sin\frac{\pi}{14} >1.3,
\end{split}
\]
so $y_0>\frac{7}{4\pi}$. Now
\[
\begin{split}
f(\frac{7}{4\pi})-f(\frac{4}{5\pi})&=\frac{7}{4\pi}\sin\frac{4}{7}\pi+\frac{4}{5\pi}\sin\frac{\pi}{4}\\
&=\frac{7}{4\pi}\sin\frac{3}{7}\pi+\frac{2\sqrt{2}}{5\pi} \\
&<\frac{2.28}{\pi}.
\end{split}
\]
On the other hand,
\[
2.4\cdot(\frac{7}{4\pi}-\frac{4}{5\pi})=\frac{ 2.28}{\pi},
\]
thus
\[
f(\frac{7}{4\pi})-f(\frac{4}{5\pi})<2.4\cdot(\frac{7}{4\pi}-\frac{4}{5\pi}).
\]
Repeat the above argument, we get
\[
f(y_0)-f(x_0) <2.4\cdot(y_0-x_0).
\]
By \eqref{eqn:2.12} again we have
\begin{equation}\label{eqn:2.14}
f'(x_0)=f'(y_0)<1.2.
\end{equation}

Assume $\psi(x_0,y_0)\geqslant\sqrt{2}$, using \eqref{eqn:2.12} we have
\begin{equation}\label{eqn:2.15}
f'(x_0)=f'(y_0)\geqslant\frac{1}{\sqrt{2(y_0-x_0)}}\geqslant\frac{1}{f(y_0)-f(x_0)}.
\end{equation}
Since $y_0<\frac{4}{\pi}$, by \eqref{eqn:2.15} we have
\[
f'(y_0)\geqslant\sqrt{\frac{\pi}{8}}.
\]
Using this one can improve the upper bound of $y_0$. In fact,
\[
f'(\frac{3}{\pi})=\sin\frac{\pi}{3}-\frac{\pi}{3}\cos\frac{\pi}{3}=\frac{\sqrt{3}}{2}-\frac{\pi}{6}<\sqrt{\frac{\pi}{8}}.
\]
So $y_0<\frac{3}{\pi}$, and $f'(y_0)\geqslant\sqrt{\frac{\pi}{6}}$. By Lemma \ref{lem:1.5}, we have
\[
f'(\frac{2.5}{\pi})=\sin\frac{2}{5}\pi-\frac{2}{5}\pi\cos\frac{2}{5}\pi<\frac{1}{3}(\frac{2}{5}\pi)^3<\sqrt{\frac{\pi}{6}},
\]
thus $y_0<\frac{2.5}{\pi}$. By Lemma \ref{lem:1.1} and Taylor's expansion formula, it's not difficult to verify that
\[
f(y_0)-f(x_0)<\frac{2.5}{\pi}\sin\frac{2}{5}\pi+\sin\theta_1<1.
\]
From \eqref{eqn:2.15} we have $f'(y_0)>1$. It follows that $y_0<\frac{2}{\pi}$, since $f'(\frac{2}{\pi})=1$. Now we estimate $x_0$. Note that
\[
\begin{split}
f'(\frac{0.7}{\pi})&=\sin\frac{10}{7}\pi-\frac{10}{7}\pi\cos\frac{10}{7}\pi\\
&=-\cos\frac{\pi}{14}+\frac{10}{7}\pi\cdot\sin\frac{\pi}{14}.
\end{split}
\]
It's easy to verify that $0<f'(\frac{0.7}{\pi})<1$. Thus $x_0>\frac{0.7}{\pi}$, and
\[
f(y_0)-f(x_0)<\frac{2}{\pi}+\frac{0.7}{\pi}=\frac{2.7}{\pi}.
\]
By \eqref{eqn:2.15} again, we have $f'(y_0)>\frac{\pi}{2.7}> 1.16$. By the mean value theorem, we have the following estimate
\[
\begin{split}
f'(\frac{1.9}{\pi})&\leqslant f'(\frac{2}{\pi})+|f''(\xi)|(\frac{2}{\pi}-\frac{1.9}{\pi})\\
&=1+\frac{0.1}{\pi}\cdot\frac{1}{\xi^3}\sin\frac{1}{\xi}\\
&<1+\frac{0.1}{\pi}\cdot(\frac{\pi}{1.9})^3<1.16.
\end{split}
\]
So $y_0<\frac{1.9}{\pi}$, and
\[
f(y_0)-f(x_0)<\frac{1.9}{\pi}+\frac{0.7}{\pi}=\frac{2.6}{\pi}.
\]
By \eqref{eqn:2.15} again, we have
\[
f'(y_0)>\frac{\pi}{2.6}>1.2,
\]
which contradicts with \eqref{eqn:2.14}.
\end{proof}

Finally, we can prove the following main theorem.

\begin{thm}\label{thm:2.4}
Let $f(x)=x\cdot\sin\frac{1}{x}$, $x\in(0,\infty)$. Then
\[
|f(y)-f(x)|\leqslant\sqrt{2|y-x|},\quad\forall\ x,y\in(0,\infty).
\]
\end{thm}
\begin{proof}
Define $J_0=[\frac{1}{\alpha_1},\infty)$, $J_n=[\frac{1}{\alpha_{n+1}},\frac{1}{\alpha_n})$, $n=1,2,\ldots$. From \eqref{eqn:2.1} we see that $f$ is monotone in each interval $J_k$, $k=0,1,\ldots$. Moreover,
\[
f(\frac{1}{\alpha_n})=\frac{1}{\alpha_n}\cdot\sin\alpha_n=\cos\alpha_n=(-1)^n\cdot\sin\theta_n.
\]
By Lemma \ref{lem:1.3}, we have
\begin{equation}\label{eqn:2.16}
f(J_0)\supset f(J_1)\supset\cdots\supset f(J_k)\supset\cdots.
\end{equation}
Now suppose $y>x>0$. If $y\in J_k$, then $x\in J_{\ell}$ for some $\ell\geqslant k$. By \eqref{eqn:2.16}, one can also choose $x',y'\in J_m$ for some $\ell\geqslant m \geqslant k$, 
such that $f(x')=f(x)$ and $f(y')=f(y)$. By Propositions \ref{prop:2.1}, \ref{prop:2.2} and \ref{prop:2.4}, we have
\[
|f(y)-f(x)|=|f(y')-f(x')|\leqslant\sqrt{2(y'-x')}\leqslant\sqrt{2(y-x)}.
\]
\end{proof}


\noindent{\bf Acknowledgments :} The second author would like to thank Prof. Youde Wang for his invitation to visit AMSS and his hospitality. Some of computations in this work are completed during the second author's visiting AMSS.


\bigskip

\noindent Jiaqiang Mei\\
Department of Mathematics\\
Nanjing University\\
Jiangsu China 210093\\
meijq@nju.edu.cn\\
\medskip

\noindent Haifeng Xu\\
School of Mathematical Sciences\\
Yangzhou University\\
Jiangsu China 225002\\
hfxu@yzu.edu.cn\\


\end{document}